\documentclass[12pt]{amsart}
\usepackage{amssymb, amsmath, amsthm, amsfonts, amscd}
\usepackage{xcolor}
\usepackage[sort, numbers]{natbib}
\usepackage{datetime}
\usepackage{hyperref}
\hypersetup{
	colorlinks=true,
	linkcolor=[rgb]{0.64, 0.0, 0.0},
	filecolor=magenta,      
	urlcolor=[rgb]{0.0,0.42,0.1},
	citecolor= [rgb]{0.0,0.32,0.89},
}
\usepackage{mathtools}
\usepackage{mathabx}
\usepackage{enumerate}
\usepackage[mathscr]{euscript}
\usepackage{url}
\usepackage{enumitem}
\usepackage{mathtools}
\usepackage{commath}
\definecolor{cobalt}{rgb}{0.0, 0.28, 0.67}
\usepackage[capitalize,nameinlink]{cleveref}
\bibliography{myrefs}

\newtheorem{theorem}{Theorem}[section]

\newtheorem{proposition}[theorem]{Proposition}
\newtheorem{corollary}[theorem]{Corollary}
\theoremstyle{definition}

\newtheorem{example}[theorem]{Example}

\theoremstyle{remark}
\newtheorem{remark}[theorem]{Remark}

\newcommand{\N}{\mathbb{N}}
\newcommand{\D}{\mathbb{D}}
\newcommand{\C}{\mathbb{C}}
\newcommand{\HE}{\mathcal{H(E)}}
\newcommand{\E}{\mathcal{E}}
\newcommand{\clh}{\mathcal{H}}
\newcommand{\NA}{\mathcal{NA}}
\numberwithin{equation}{section}

\title[Norm attaining composition operators on Segal-Bargmann spaces]{Norm attaining composition operators on Segal-Bargmann spaces}

\author{Neeru Bala }
\address{Department of Mathematics and computing, Indian Institute of Technology (ISM), Dhanbad 826004, India}
\email{\textcolor{cobalt}{neerusingh41@gmail.com, neerubala@iitism.ac.in}}
\author{Sudip Ranjan Bhuia }
\address{Indian Statistical Institute, Statistics and Mathematics Unit, 8th Mile, Mysore Road, Bangalore, 560059, India}
\email{\textcolor{cobalt}{sudipranjanb@gmail.com}}

\subjclass[2010]{  47B38, 47B33 }
\subjclass[2020]{Primary 47B38 ; Secondary 47B33}
\keywords{Segal-Bargmann spaces, composition operators, norm attaining operators, reproducing kernel Hilbert spaces}
\date{\currenttime ;  \today}

\begin{document}
	
	\maketitle
\begin{abstract}
In this note, we study the composition operators on Segal-Bargmann spaces, which attains its norm and we show that every composition operators on the classical Fock space over $\mathbb{ C}^n$ is norm attaining. Also, we establish a necessary and sufficient condition for a sum of two kernel functions to be an extremal function for the norm of composition operators.
\end{abstract}
\section{Introduction}
In this article, we aim to merge two classical notions in operator theory: norm attaining property and composition operators.

The study of norm attaining operators is motivated by norm attaining functionals, which date back to the Hahn-Banach theorem or even earlier. Two of the well explored results in functional analysis are the Hahn-Banach theorem and the Bishop-Phelps theorem.
The first proves the presence of non-zero norm-attaining functionals in the dual of a Banach space, while the second proves the denseness of norm attaining functionals in the dual of a Banach space.
Norm attaining functionals are also important in the analysis of the underlying space; for example, the James theorem states that a Banach space $X$ is reflexive if and only if every bounded linear functional or every compact operator on $X$ is norm-attaining.

Let $H$ be a infinite dimensional complex Hilbert space and $\mathcal{B}(H)$ be the space of all bounded linear operators on $H$. Then $T\in \mathcal{B}(H)$ is said to be norm-attaining if there exists a non-zero unit vector $x \in H$ such that
\begin{equation}
\left\|Tx\right\|=\left\|T\right\|
\end{equation}
and such an element $x$ is called the extremal point for $\|T\|$. Throughout this article, $T\in \NA$ means $T$ is norm attaining. If $H$ is finite dimensional, then every $T\in \mathcal{B}(H)$ is norm attaining. Also compact operators and isometries are norm attaining. An operator $T\in \mathcal{B}(H)$ with $\|T\|_e<\|T\|$ is norm attaining, where $\|T\|_e$ is the essential norm of $T$.

 Norm attaining property of operators is connected to several different concepts in mathematics, for example Radon-Nikodym property \cite{CHOI} and reflexivity. Norm attaining operators has been studied from different perspectives, for example Bishop-Phelps-Bollobas property \cite{AL,ACOSTA, CASCALES}, invariant subspace of some non-normal operators \cite{LEE, IST} and to study the spectrum of operators \cite{KOV1,KOV2}.

Let $\mathcal{B}$ be a Banach space of function on a set $\mathscr{X}$, and $\varphi:\mathscr{X}\rightarrow\mathscr{X}$ be a mapping. Then define the composition operator $C_\varphi$ by $C_\varphi h=h\circ\varphi$ for any function $h\in \mathcal{B}$ for which the function $h\circ\varphi$ also belongs to $\mathcal{B}$.

Composition operator plays a significant role in operator theory and function theory, for example the invariant subspace problem is directly related to the existence of eigenvalue of composition operators on the Hardy space.

Norm-attaining composition operators have been studied for different function spaces by several authors, for example, the Hardy space and the Dirichlet space by Hammond \cite{Hammond:FLT,Hammond:DS}, Bloch spaces by Mart\'{i}n \cite{Martin} and Montes-Rodr\'{i}guez \cite{Motes}, and weighted Bloch spaces by Bonet, Lindstr\"{o}m and Wolf \cite{Bonet:Wolf}. In \cite{Appel:Bourdon}, the authors have proved that the normalized reproducing kernels are not necessarily the extremal functions for $\|C_\varphi\|$ in the classical Hardy space which answers a question posed by Cowen and MacCluer in \cite[p. 125]{Cowen book}. In \cite{Martin}, the authors have proved that every composition operator $C_\varphi$ on the Bloch space (modulo constant functions) attains its norm and this is quite interesting fact and this motivates to ask the similar question for the composition operators defined on the Segal-Bargmann spaces, in particular, composition operators defined on the Fock space (see \cite{Kehe Zhu book}). T. Le in \cite{Tle}, characterized the bounded and compact composition operators $C_\varphi$ defined on the Segal-Bargmann spaces $\HE$, where $\E$ is any infinite dimensional complex Hilbert space. In fact, the authors have shown that $C_\varphi$ is bounded if and only if $\varphi(z)=Az+b$, where $A$ is a linear operator defined on $\E$ with $\|A\|\leq 1$ and $A^*b$ belongs to the range of $(I-A^*A)^{1/2}$ (cf. \cref{bnd-cpt thm1}). Very recently, the dynamical properties of composition operators on the Segal-Bargmann space have been studied by G. Ramesh, the second author, and D. Venku Naidu in \cite{Ramesh:Sudip:Venku}.

We investigate whether composition operators $C_\varphi$ acting on the Segal-Bargmann space achieve their norms in this study.  As a consequence, we are very fortunate to show that 
\begin{center}
	\emph{"Every composition operator $C_\varphi$ on the Fock space $\mathcal{H}(\mathbb{ C}^n)$ attains its norm"}.
\end{center}
Since the linear span of the kernel functions is dense in $\HE$, it is quite natural to ask when a kernel function becomes an extremal function of the norm of a composition operator $C_\varphi$ defined on $\HE$. Interestingly, we are able to prove that if $C^*_\varphi$ attains its norm at every normalized kernel functions, then the linear operator $A$ associated to the symbol $\varphi$ is isometry. In fact it is necessary and sufficient.

In this article, we have used the following identity frequently which applies to all space and appears to be quite powerful:
\begin{equation}
C^*_\varphi K_w=K_{\varphi(w)},
\end{equation}
where $K_w$ is the reproducing kernel for $w\in \E$.

This article is organized as follows: the second section contains some preliminary results that will be used in subsequent sections. In the third section, we have shown that $\NA$ property of the linear operator $A$ on $\E$ influences the $\NA$ property of $C_\varphi$ on $\HE$ and vice versa. In the later part of this section, we have show that every composition operator on the Fock space $\mathcal{H}(\mathbb{ C}^n)$ is $\NA$. In the fourth section, we find a necessary and sufficient condition for sum of two kernel functions to be an extremal function for the norm of $C_\varphi$.
\section{Preliminaries}
We begin this section with one of the fundamental results on norm attaining operators.
\begin{theorem}
	Let $T \in \mathcal{B}(H)$. Then the following are equivalent:
	\begin{enumerate}
		\item $T \in \mathcal{NA}$.
		\item $T^* \in \mathcal{NA}$.
		\item $TT^* \in \mathcal{NA}$.
		\item  $\|T\|^2$ is in the point spectrum of $T T^*$.
	\end{enumerate}
\end{theorem}
The following result will help us to realize the elements of the space $\HE$. For more detailed construction of this space, we refer \cite[Section 2.1]{Tle}.
\begin{proposition}\cite{Tle}\label{cyclicp1.1}
	Each element $f$ in $\HE$ can be identified as an entire function on $\E$ having a power series expansion of the form $$f(z)=\sum_{j=0}^{\infty}\langle z^j,a_j\rangle \hspace{1cm}\text{for all}\,\, z\in \E,$$  where $ a_j\in \E^j,\; j=0,1,2,\dots$.
	Furthermore,
	$\left\|f\right\|^2 = \displaystyle \sum_{j=0}^{\infty}j!\left\|a_j\right\|^2.$
	
	Conversely, if $\displaystyle \sum_{j=0}^{\infty}j!\left\|a_j\right\|^2 <\infty$, then the power series $\displaystyle \sum_{j=0}^{\infty}\langle z^j,a_j\rangle$ defines an element in $\HE$.
\end{proposition}
The function
\begin{equation*}
K(z,w):=K_{w}(z)=\exp\langle z,w\rangle\; \text{for all}\; z,w\in \E,
\end{equation*}
 is the reproducing kernel function for $\HE$ and the normalized kernel function is defined by
\begin{equation*}
k_w(z)= \exp\left(\langle z,w\rangle-\frac{\left\|w\right\|^2}{2}\right).
\end{equation*}
The linear span of the set $\{K_{w} : w \in \E\}$ is dense in $\HE$. As a result, $\HE$ is a reproducing kernel Hilbert space. For each $f\in \HE $, we have $\langle f, K(x,\cdot)\rangle=f(x)$ for all $x\in \E$. For more details on these spaces, see Chapter 2 of \cite{Paulsen}.
\begin{theorem}\label{bnd-cpt thm1}\cite[Theorem 1.3]{Tle}
	Let $\varphi:\E_1\rightarrow \E_2$ be a mapping. Then the composition operator $C_{\varphi}:\mathcal{H}(\E_2) \rightarrow \mathcal{H}(\E_1)$ is bounded if and only if $\varphi(z)=Az+b$ for all $z\in \E_1 $, where $A:\E_1 \rightarrow \E_2$ is a bounded linear operator with $\left\|A\right\|\leq1$ and $A^*b$ belongs to the range of $(I-A^*A)^{\frac{1}{2}}$.
	Furthermore, the norm of $\left\|C_{\varphi}\right\|$ is given by
	$$\left\|C_{\varphi}\right\|=\exp\Big(\frac{1}{2}\left\|v\right\|^2 + \frac{1}{2}\left\|b\right\|^2\Big),$$
	where $v$ is the unique vector in $\E_1$ of minimum norm satisfying $A^*b=(I-A^*A)^{\frac{1}{2}}v$.
\end{theorem}
\begin{theorem}\label{cyclict-1}\cite[Theorem 3.7]{Tle}
	Let $\varphi:\E_1\rightarrow \E_2$ be a mapping. Then the composition operator $C_{\varphi}:\mathcal{H}(\E_2) \rightarrow \mathcal{H}(\E_1)$ is bounded if and only if there is a bounded linear operator $A:\E_1 \rightarrow \E_2$ with $\left\|A\right\|\leq1$ and a vector $b$ in  the range of $(I-AA^*)^{\frac{1}{2}}$ such that $\varphi(z)=Az+b$ for all $z\in \E_1 $.
	Furthermore, the norm of $\left\|C_{\varphi}\right\|$ is given by
	$$\left\|C_{\varphi}\right\|=\exp\Big(\frac{\left\|u\right\|^2 }{2}\Big),$$
	where $u$ is the unique vector in $\E_2$ of minimum norm that satisfies the equation $b=(I-AA^*)^{\frac{1}{2}}u$.
\end{theorem}
\begin{theorem}\cite[Theorem 1.5]{Tle}
	Let $\varphi:\E_1\rightarrow \E_2$ be a mapping. Then the composition operator $C_{\varphi}:\mathcal{H}(\E_2) \rightarrow \mathcal{H}(\E_1)$ is compact if and only if there is a compact linear operator $A:\E_1\rightarrow\E_2$ with $\left\|A\right\|<1$ and a vector $b\in \E_2$ such that $\varphi(z)=Az+b$ for all $z\in \E_1$.
\end{theorem}
For $\E_1=\E_2=\mathbb{ C}^n$, the boundedness and compactness of $C_\varphi$ are discussed by Carswell, MacCluer, and Schuster in \cite{carswell et al.}.
\begin{remark}
	It is clear that if $\|A\|<1$ and $A$ is compact, then $C_\varphi$ is compact and hence norm attaining operator.
\end{remark}
\begin{theorem}\cite{carswell et al.}
	Let $\varphi : \mathbb{C}^n\rightarrow \mathbb{C}^n$ be a holomorphic mapping. Then the following statements hold:
	\begin{enumerate}
		\item $C_{\varphi}$ is bounded on $\mathcal{H}(\mathbb{C}^n)$ if an only if $\varphi(z)=Az+B$ for some $n\times n$ matrix $A$ with $\left\|A\right\|\leq 1$ and $n \times 1$ vector $B$ such that $\langle A\zeta , B\rangle=0 $ whenever $\zeta \in \mathbb{C}^n$ and $|A\zeta|=|\zeta|;$
		\item $C_{\varphi}$ is compact on $\mathcal{H}(\mathbb{C}^n)$ if and only if $\varphi(z)=Az+B$ for some $n\times n$ matrix $A$ with $\left\|A\right\|< 1$ and $n \times 1$ vector $B$.
		\item $\left\|C_{\varphi}\right\| = \exp(\frac{1}{2}(|w_0|^2 -|Aw_0|^2 +|B|^2))$, where $w_0$ is the solution of the equation $(I-A^*A)z=A^*B.$
	\end{enumerate}
Here $|w|=\left(\displaystyle\sum_{i=1}^{n}|w_i|^2\right)^{1/2}$, for a vector $w=(w_1,w_2,\dots,w_n)\in \mathbb{ C}^n$.
\end{theorem}
\section{Norm attaining composition operators on $\HE$}
In this section, we investigate the norm attaining composition operators on the Segal-Bargmann spaces. Before we proceed to do so, we start with the following convention for our investigation.

We say that a map $\varphi$ on $\E$ has the property $\mathcal{P}$, if it satisfies the following:
\begin{enumerate}[label=(\roman*)]
	\item  $\varphi(z)=Az+b$ for all $z\in \E $
	\item  $A:\E \rightarrow \E$ is a bounded linear operator with $\left\|A\right\|\leq1$ and $b\in \E$
	\item $A^*b$ belongs to the range of $(I-A^*A)^{1/2}$ and $v$ is the unique element in $\E$ of smallest norm such that $(I-A^*A)^{1/2}v=A^*b$.
\end{enumerate}

\begin{remark}
	Let $\varphi:\E\rightarrow \E$ be a mapping. Then $\varphi$ satisfies the property $\mathcal{P}$ if and only if the induced composition operator $C_{\varphi}$ is bounded on the Segal-Bargmann space $\HE$.
\end{remark}
\begin{proposition}
	Let $\varphi$ be a mapping on $\E $ satisfying the property $\mathcal{P}$ such that $\varphi(0)=0$. Then $C_\varphi\in \mathcal{NA}$.
\end{proposition}
\begin{proof}
	First we note that
	\begin{equation*}
	C_\varphi K_0=C_\varphi 1=1=K_0.
	\end{equation*}
	That is, $1$ is an eigenvalue of $C_\varphi$. Since $\varphi(0)=0$ and $C_\varphi$ is bounded composition operator on $\mathscr{H(E)}$, by \cref{bnd-cpt thm1}, we have $\varphi(z)=Az$ for all $z\in \E$ with $\left\|A\right\|\leq 1$ and the norm formula gives $\left\|C_\varphi\right\|=1$. Therefore, we conclude that $\left\|C_\varphi\right\|$ belongs to the point spectrum of $C_\varphi C^*_\varphi$, and hence $C_\varphi$ is norm-attaining.
\end{proof}

\begin{theorem}
Let $\varphi$ be a mapping on $\E $ satisfying the property $\mathcal{P}$ such that $A^*b=0$ and $w_0\in \E$ with $\|w_0\|=1$. Then $A$ attains its norm at $w_0$ and $\|A\|=1$ if and only if $C^*_\varphi$ attains its norm at $k_{w_0}$.
\end{theorem}
\begin{proof}
First, we assume that $A$ attains norm at $w_0$ and $\|A\|=1$. Therefore, we have $\|Aw_0\|=\|A\|=1$. Now consider the normalized kernel function $k_{w_0}=\frac{K_{w_0}}{\|K_{w_0}\|}$, then we see that 
\begin{equation}
\left\|C^*_\varphi\left(k_{w_0}\right) \right\|^2=\exp\left(\left\|\varphi(w_0)\right\|^2-\|w_0\|^2\right)=
\exp\left(\left\|Aw_0\right\|^2+\left\|b\right\|^2-\|w_0\|^2\right)=\exp\left(\left\|b\right\|^2\right)
\end{equation}
and this implies that $C^*_\varphi$ attains norm at $k_{w_0}$.

Conversely, if $C^*_\varphi$  attains its norm at the kernel function $k_{w_0}$, then we have
	\begin{equation}
	\begin{split}
	C_{\varphi} C_{\varphi}^{*} K_{w_{0}}=\left\|C_{\varphi}\right\|^{2} K_{w_{0}},
	\end{split}
	\end{equation}
	and the norm formula implies that $\langle Az,Aw_0\rangle=\langle z,w_0\rangle$ for all $z\in \E$. In particular, we have $\|Aw_0\|=\|w_0\|=1$. Since $\|A\|\leq 1$ we have $\|A\|=1$. Hence we conclude that $A$ attains its norm at $w_0$ and $\|A\|=1$.
\end{proof}

Next we establish a necessary and sufficient condition for a normalized kernel function to be a norm attaining function for $C^*_\varphi$.
\begin{theorem}\label{NA:Adjoint:Kernel function}
Let $\varphi$ be a mapping on $\E $ satisfying the property $\mathcal{P}$. Then the following are true:
\begin{enumerate}
	\item\label{main thm1}  Let  $z\in \E$. Then $C^*_{\varphi}$ attains norm at $\frac{K_{z}}{\left\|K_{z}\right\|}$ if and only if  $v =\left(1-A^{*} A\right)^{1/2}z$.
	\item \label{main thm2} $C_\varphi$ is norm attaining provided $v$ is in the range of $(I-A^*A)^{\frac{1}{2}}$.
	\item\label{main thm3} $C^*_\varphi$ attains its norm at $\frac{K_z}{\|K_z\|}$ for every $z\in \E$ if and only if the linear operator $A$ is an isometry on $\E$.
\end{enumerate}

\end{theorem}
\begin{proof}
	Let $\varphi$ be a mapping on $\E $ satisfying the property $\mathcal{P}$. Then for all $z\in \E$ we have
	\begin{equation}\label{Important Equtn}
		\left\|\varphi(z)\right\|^2=\left\|Az+b\right\|^2=-\left\|(I-A^*A)^{1/2}z-v\right\|^2+\|v\|^2+\|b\|^2+\|z\|^2.
	\end{equation}
{\bf Proof of (\ref{main thm1}):}	Since $C^*_\varphi$ attains norm at the normalized kernel function $k_z$, we have
	\begin{equation}
	\begin{split}
	\left\|C_{\varphi}^{*} \frac{K_{z}}{\left\|K_{z}\right\|}\right\|^2&=\left\|C_{\varphi}\right\|^2\,\text{or}\\
	 \frac{\|K_{\varphi(z)}\|^2}{\left\|K_{z}\right\|^2}&=\|C_ \varphi\|^2.
	\end{split}
\end{equation}
That is,
	 	\begin{equation}
	 \exp \left(\|\varphi(z)\|^{2}-\|z\|^{2}\right)=\exp \left(\|b\|^{2}+\|v\|^{2}\right).
\end{equation}
	By using Equation \ref{Important Equtn}, we get
	 \begin{equation}
	 \left\|\left(1-A^{*} A\right)^{1 / 2} z-v\right\|^{2}=0.
	\end{equation}
	Thus $C^*_\varphi$ attains norm at $k_z$ if and only if $z$ satisfies $v =\left(1-A^{*} A\right)^{1 /2}z$.

{\bf Proof of (\ref{main thm2}):} Since $v$ belongs to the range of $(I-A^*A)^{\frac{1}{2}}$, set $v=(I-A^*A)^{\frac{1}{2}}w$ for some $w\in \E$. Then $A^*b=(I-A^*A)^{\frac{1}{2}}v=(I-A^*A)w$.
	
	Therefore, by using the expression as in Equation \ref{Important Equtn}, we get
	\begin{equation*}
		\begin{split}
			\left\|C^*_\varphi\left(\frac{K_w}{\left\|K_w\right\|}\right)\right\|^2&=\exp\left(\left\|\varphi(w)\right\|^2-\|w\|^2 \right)\\&=\exp\left(-\left\|\left(I-A^{*} A\right)^{1 / 2} w-v\right\|^{2}+\|v\|^{2}+\|b\|^{2}\right)\\&=\exp (\|v\|^{2}+\|b\|^{2})=\left\|C^*_\varphi\right\|^2.
		\end{split}
	\end{equation*}
	So $C^*_\varphi$ is norm attaining and hence $C_\varphi$.

{\bf Proof of (\ref{main thm3}):} First we assume that $A$ is isometry. By \cref{bnd-cpt thm1}, we have $A^*b=0$ and by the norm formula we have $\left\|C_\varphi\right\|^2=e^{\left\|b\right\|^2}$ . Let $z\in \E$ be arbitrary. Consider the normalized kernel function $\frac{K_z}{\|K_z\|}\in \HE$. Then 
	\begin{equation}\label{isometry_NA}
		\left\|C^*_\varphi\left(\frac{K_z}{\|K_z\|}\right) \right\|^2=\exp\left(\left\|\varphi(z)\right\|^2-\|z\|^2\right)=
		\exp\left(\left\|Az\right\|^2+\left\|b\right\|^2-\|z\|^2\right)=\exp\left(\left\|b\right\|^2\right).
	\end{equation}
	Hence from \cref{isometry_NA}, we get $\left\|C^*_\varphi \left(\frac{K_z}{\left\|K_z\right\|}\right)\right\|=e^{\frac{\left\|b\right\|^2}{2}}=\left\|C_\varphi\right\|$ and this implies that $C^*_\varphi$ attains its norm at $\frac{K_z}{\left\|K_z\right\|}$ for every $z\in \E$.

	Next, we assume that $C^*_\varphi$ attains norm at $\frac{K_z}{\left\|K_z\right\|}$ for every $z\in \E$. Then we have
	\begin{equation}
		\begin{split}
			\left\|C_{\varphi}^{*} \frac{K_{z}}{\left\|K_{z}\right\|}\right\|^2&=\left\|C_{\varphi}\right\|^2\,\text{or,}\\
			\frac{\|K_{\varphi(z)}\|^2}{\left\|K_{z}\right\|^2}&=\|C_ \varphi\|^2.
		\end{split}
	\end{equation}
	That is,
	\begin{equation}
			\exp \left(\|\varphi(z)\|^{2}-\|z\|^{2}\right)=\exp \left(\|b\|^{2}+\|v\|^{2}\right).
	\end{equation}

	Using Equation \ref{Important Equtn}, we get
	\begin{equation}
		\left\|\left(1-A^{*} A\right)^{1/2} z-v\right\|^{2}=0.
	\end{equation}
	Hence $\left(1-A^{*} A\right)^{1/2}z=v$ for every $z\in \E$. In particular, for $z=0$, we have $v=0$. Thus
	\begin{equation*}
		(I-A^*A)z=0\quad \text{for all}\,z\in \E.
	\end{equation*}
	Thus we conclude that $A$ is isometry. This completes the proof.

\end{proof}

\begin{corollary}
	Let $\varphi$ be a mapping on $\E $ satisfying the property $\mathcal{P}$ with $\|A\|<1$. Then the bounded composition operator $C_\varphi$ is $\NA$.
\end{corollary}
\begin{proof}
	Since $C_\varphi$ is bounded and $v$ is the smallest norm vector in $\E$ such that $A^*b=(I-A^*A)^{1/2}v$. As $\|A\|<1$, the operator $(I-A^*A)^{1/2}$ is invertible and hence (\ref{main thm2}) of Theorem \ref{NA:Adjoint:Kernel function}, ensures that $C_\varphi$ is $\NA$.
\end{proof}

\begin{remark}
	By \cite[Proposition 4.1]{Tle}, it is clear that if $\|C_\varphi\|_e<\|C_\varphi\|$, then $\|A\|<1$ and hence $C_\varphi$ is $\NA$.
\end{remark}
\begin{corollary}
	Every composition operator on the Fock space $\mathcal{H}(\mathbb{ C}^n)$ attains their norm.
\end{corollary}
\begin{proof}
	Let $C_\varphi$ be a bounded composition operator on $\mathcal{H}(\mathbb{ C}^n)$. Then \cref{bnd-cpt thm1}, we have $\varphi(z)=Az+b$, $\|A\|\leq 1$ and $v$ is the vector of smallest norm such that $A^*b=(I-A^*A)^\frac{1}{2}v$ and $\left\|C_\varphi\right\|=\exp\left(\frac{\|v\|^2+\|b\|^2}{2}\right)$. Then by \cite[Remark 3.2]{Tle} and (\ref{main thm2}) of Theorem \ref{NA:Adjoint:Kernel function}, the conclusion follows.
\end{proof}

\begin{remark}
Under the same hypothesis in the Theorem \ref{NA:Adjoint:Kernel function} and using the fact that $\ker(I-A^*)\subset\ker(I-A^*A)^{1/2}$ for any complex Hilbert space operator with $\|A\|\leq 1$ (cf. \cite{Ramesh:Sudip:Venku}), we can conclude that $v \in \overline{\text{ran}}\left(I- A\right)$, the closure of the range of $(I-A)$.
\end{remark}
Suppose that $C_\varphi$ attains norm at $g\in \HE$, then we have $C^*_\varphi C_\varphi g=\left\|C_\varphi\right\|^2g$. Therefore,
\begin{equation}\label{observation:NA}
\begin{split}
\left\|C_\varphi\right\|^2g(0)=\langle \left\|C_\varphi\right\|^2g,K_0\rangle=\langle C^*_\varphi C_\varphi g,K_0\rangle=\langle C_\varphi g,K_0\rangle=g(\varphi(0)).
\end{split}
\end{equation}
\begin{proposition}\label{positivity}
	Let $\varphi$ be a mapping on $\E $ satisfying the property $\mathcal{P}$. Then the following are true:
	\begin{enumerate}

		\item\label{postv2} If $C_\varphi$ attains norm at $K_b$, then $A^*b=0$.
		\item\label{postv3} Let $f\in \HE$ be a non zero function with $f
		(0)\neq 0$ such that $C_\varphi C^*_\varphi f=\left\|C_\varphi\right\|^2f$. Then $\frac{f(A^*b)}{f(0)}\geq 0$.
		\item\label{postv4} If $f\in \HE$ with $f(0)\neq 0$ such that $C^*_\varphi C_\varphi f=\left\|C_\varphi\right\|^2f$. Then  $\frac{f(b)}{f(0)}\geq 0$.
	\end{enumerate} 
\end{proposition}
\begin{proof}

	{\bf Proof of (\ref{postv2}):} 	The operator $C_\varphi$ is bounded, the norm is given by 	$$\left\|C_{\varphi}\right\|=\exp\Big(\frac{1}{2}\left\|v\right\|^2 + \frac{1}{2}\left\|b\right\|^2\Big),$$
	where $v$ is the unique vector in $\E_1$ of minimum norm satisfying $A^*b=(I-A^*A)^{\frac{1}{2}}v$. Since $C_\varphi$ attains norm at $K_b$, by \cref{observation:NA}, we have
	\begin{equation}
	\begin{split}
	K_b(b)&=e^{\|v\|^2+\|b\|^2}K_b(0),\,\text{which implies}\\
	e^{\|b\|^2}&=e^{\|v\|^2+\|b\|^2}.
	\end{split}
	\end{equation}
	Thus we obtain $v=0$, and hence by norm formula, we have $A^*b=0$.

	{\bf Proof of (\ref{postv3}):} Let $f\in \HE$ be a non zero function with $f
	(0)\neq 0$ such that $C_\varphi C^*_\varphi f=\left\|C_\varphi\right\|^2f$. Since linear span kernel functions is dense in $\HE$, we write $f$ as $f=\displaystyle\sum_{i}s_iK_{x_i}$, where $x_i\in \E$. Therefore, we have 
	\begin{equation*}
	\begin{split}
	C_\varphi C^*_\varphi\displaystyle\sum_{i}s_iK_{x_i}&=\left\|C_\varphi\right\|^2f.
	\end{split}
	\end{equation*}
	By taking inner product both sides of the above equation with $K_0$, the kernel function at $0$, we get
	\begin{equation}
	\begin{split}
	\displaystyle\sum_{i}s_i\langle K_{\varphi(x_i)},K_{\varphi(0)}\rangle&=\left\|C_\varphi\right\|^2f(0)\,\text{or}\\
	\displaystyle\sum_{i}s_i\exp(\|b\|^2)\exp\langle A^*b,x_i\rangle&=\exp(\|v\|^2+\|b\|^2)f(0).
\end{split}
\end{equation}
That is, \begin{equation}
	\displaystyle\sum_{i}s_i\exp\langle A^*b,x_i\rangle=\exp(\|v\|^2)f(0).
\end{equation}
Thus we have
	\begin{equation}
		\begin{split}
	\left\langle \displaystyle\sum_{i}s_iK_{x_i},K_{A^*b}\right\rangle&=\exp(\|v\|^2)f(0),\,\text{which implies}\\
	f(A^*b)&=\exp(\|v\|^2)f(0).
	\end{split}
	\end{equation}
	This shows that $\frac{f(A^*b)}{f(0)}\geq 0$.
	
	{\bf Proof of (\ref{postv4}):} From \cref{observation:NA}, we get $\left\|C_\varphi\right\|^2f(0)=f(b)$ and this will imply $\frac{f(b)}{f(0)}\geq 0$.
\end{proof}
The following is an easy consequence of the above lemma:
\begin{corollary}
	Let $\varphi$ be a mapping on $\E $ satisfying the property $\mathcal{P}$. If $C_\varphi$ attains its norm at the kernel function $\frac{K_w}{\|K_w\|}$, then the following are true:
	\begin{enumerate}
	 \item\label{postv1} If $C_\varphi$ attains its norm at the normalized kernel function $\frac{K_w}{\|K_w\|}$, then $\langle b,w\rangle=\|v\|^2+\|b\|^2\geq 0$. Moreover,  $\left\|C_\varphi\right\|=\exp\frac{\langle b,w\rangle}{2}$.
		
		\item If $C^*_\varphi$ attains its norm at the normalized kernel function $\frac{K_w}{\|K_w\|}$, then $\langle A^*b,w\rangle=\|v\|^2\geq 0$.
	\end{enumerate}
\end{corollary}
\begin{proof}
	Directly follows from \cref{positivity}.
\end{proof}

\section{Extremal functions}
In this section, we will investigate on the extremal function for the norm of a bounded composition operator $C_\varphi$ on the Segal-Bargmann space $\HE$.
From \cref{NA:Adjoint:Kernel function}, it is clear that the normalized kernel function $k_w$, where $w\in \E$ is an extremal function for $\|C_\varphi\|$ if and only if $w$ satisfies $(I-A^*A)^{1/2}w=v$.

\begin{remark}
	If the kernel function $\frac{K_w}{\left\|K_w\right\|}$ for some nonzero element $w\in \E$ is the extremal function for $\|C_\varphi\|$, then $w$ satisfies $(I-A^*A)w=A^*b$ and the unique vector $v$ of minimum norm can be characterized by $v=(I-A^*A)^{1/2}w$.
\end{remark}

Next we find the necessary condition for a sum of two kernel functions to be an extremal function for the norm of a bounded composition operator $C_\varphi$ on $\HE$.
\begin{proposition}\label{linCombExtremal}
	Let $\varphi$ be a mapping on $\E $ satisfying the property $\mathcal{P}$ and $x_1,x_2\in \E$ with $\|x_1\|=\|x_2\|=1$. If $C^*_\varphi$ attains norm at $\frac{K_{x_1}+K_{x_2}}{\|K_{x_1}+K_{x_2}\|}$, then $\|\varphi(x_1)\|=\|\varphi(x_2)\|$.

\end{proposition}
\begin{proof}
Since $C_\varphi$ is norm attaining at $K_{x_1}+K_{x_2}$, we have
\begin{equation}
\begin{split}
C_\varphi C^*_\varphi(K_{x_1}+K_{x_2})&=\left\|C_\varphi\right\|^2(K_{x_1}+K_{x_2}),\,\text{that is},\\
C_\varphi\left(K_{\varphi(x_1)}+K_{\varphi(x_2)}\right)&=\left\|C_\varphi\right\|^2(K_{x_1}+K_{x_2}).
\end{split}
\end{equation}
So,
\begin{equation}\label{Equtn_inner product}
	C_\varphi K_{\varphi(x_1)}-\left\|C_\varphi\right\|^2K_{x_1}=-C_\varphi K_{\varphi(x_2)}+\left\|C_\varphi\right\|^2K_{x_2}.
\end{equation}
By taking inner product both sides of the \cref{Equtn_inner product} with $K_{x_1}$, we get
\begin{equation}\label{real equtn 1}
\begin{split}
\left\langle C_\varphi K_{\varphi(x_1)},K_{x_1}\right\rangle-\left\|C_\varphi\right\|^2\langle K_{x_1},K_{x_1}\rangle&=-\left\langle C_\varphi K_{\varphi(x_2)},K_{x_1} \right\rangle+\left\|C_\varphi\right\|^2\langle K_{x_2},K_{x_1}\rangle\\
\text{or,}\,\left\|K_{\varphi(x_1)} \right\|^2-\left\|C_\varphi\right\|^2\left\|K_{x_1}\right\|^2&=-\left\langle K_{\varphi(x_2)},K_{\varphi(x_1)} \right\rangle+\left\|C_\varphi\right\|^2\langle K_{x_2},K_{x_1}\rangle.
\end{split}
\end{equation}
Similarly, by taking inner product both sides of the \cref{Equtn_inner product} with $K_{x_2}$, we get
\begin{equation}\label{real equtn 2}
-\left\|K_{\varphi(x_2)} \right\|^2+\left\|C_\varphi\right\|^2\left\|K_{x_2}\right\|^2=\left\langle K_{\varphi(x_1)},K_{\varphi(x_2)} \right\rangle-\left\|C_\varphi\right\|^2\langle K_{x_1},K_{x_2}\rangle.
\end{equation}
From Equations \ref{real equtn 1} and \ref{real equtn 2}, we get

\begin{equation}
\begin{split}
\left\|K_{\varphi(x_2)} \right\|^2-\left\|C_\varphi\right\|^2\left\|K_{x_2}\right\|^2=\left\|K_{\varphi(x_1)} \right\|^2-\left\|C_\varphi\right\|^2\left\|K_{x_1}\right\|^2.
\end{split}
\end{equation}
Since $\left\|x_1\right\|=\left\|x_2\right\|$, we have $\left\|K_{x_1}\right\|=\left\|K_{x_2}\right\|$ and consequently, we get the desired conclusion that is, $\|\varphi(x_1)\|=\|\varphi(x_2)\|$.

\end{proof}

\begin{example}
Consider the right shift operator $S$ on $\ell^2(\N)$ defined by $Se_n=e_{n+1}$ for all $n=1,2,\dots$. The vectors $e_j$ denotes the sequence whose $j$-th position is $1$ and the rest are zero. Then $S^*e_1=0$ and the composition operator $C_{\varphi}$ is bounded on $\clh (\ell^2(\N)) $with $\left\|C_{\varphi}\right\|=e^{\frac{1}{2}}$, where $\varphi(z)=Sz+e_1$. Note that $\left\|\varphi(e_3)\right\|=\left\|\varphi(e_4)\right\|$.
Also note that for $i=1,2$ we have
\begin{equation*}
	C_\varphi K_{\varphi(x_i)}=e^{\left\|b\right\|^2+\langle A^*b,x_i\rangle}K_{A^*Ax_i} K_{A^*b}.
\end{equation*}
Now
$$C_\varphi K_{\varphi(e_3)}-\left\|C_\varphi\right\|^2K_{e_3}=e^{\left\|e_1\right\|^2+\langle S^*e_1,e_3\rangle}K_{S^*Se_3} K_{S^*e_1}-eK_{e_3}=0,$$ and 
$$-C_\varphi K_{\varphi(e_4)}+\left\|C_\varphi\right\|^2K_{e_4}=-e^{\left\|e_1\right\|^2+\langle S^*e_1,e_4\rangle}K_{S^*Se_4} K_{S^*e_1}+eK_{e_4}=0.$$ Therefore, by using Equation \ref{Equtn_inner product}, we conclude that $C^*_{\varphi}$ attains norm at $\frac{K_{e_3}+K_{e_4}}{\left\|K_{e_3}+K_{e_4}\right\|}$.

\end{example}

\begin{example}
Let $\D$ denote the open unit disc in the complex plane $\C$. The space $H^2(\D)$ consists of all analytic functions on $\D$ having power series representation with square summable complex coefficients. The set $\{e_n=z^n:n\geq 0\}$ forms an orthonormal basis for $H^2(\D)$. Consider $\E=H^2(\D)$ and the operator $Af=f(0)+z(f-f(0))$ and $b=z$. Then the operator $A$ is isometry and $A^*z=0$. Similarly as in the above example one can show that the composition operator $C^*_\varphi$ on $\clh(H^2(\D))$ with $\varphi(f)=Af+z$ for all $f\in H^2(\D)$ attains norm at $\frac{K_{e_3}+K_{e_4}}{\left\|K_{e_3}+K_{e_4}\right\|}$ along with $\left\|\varphi(e_3)\right\|=\left\|\varphi(e_4)\right\|$.
\end{example}
The following example shows that the converse of the Proposition \ref{linCombExtremal} is not true in general.
\begin{example}\cite[Example 2.6]{Ramesh:Sudip:Venku}
		Let $\mu$ be a real number such  that $0<\mu\leq 1 $.  For $\{x_n\}_{n=1}^{\infty}\in \ell^2(\mathbb{N})$ define the weighted unilateral shift on $\ell^2(\mathbb{N})$ by
		\begin{equation}
			S(x_1,x_2,x_3,\dots)=(0,\mu x_1,x_2,x_3,\dots), \forall \, \{x_n\}\in \ell^2(\mathbb{N}).
		\end{equation}
		The adjoint  $S^*$ of $S$ is given by
		\begin{equation}
			S^*(x_1,x_2,x_3,\dots)=(\mu x_2,x_3,x_4,\dots), \forall \, \{x_n\}\in \ell^2(\mathbb{N}).
		\end{equation}
	
		Let $\hat{b}=(1,\frac{\sqrt{1-\mu^2}}{\mu},0,0,\dots)$.
		Then
		\begin{equation}\label{cyclicc19}
			(I-S^*S)^{\frac{1}{2}}e_1 = S^*\hat{b},
		\end{equation}
		where $e_1 =(1,0,0,\dots)$. Now consider the map $\hat{\psi}:\ell^2(\mathbb{N})\rightarrow\ell^2(\mathbb{N})$ defined by $\hat{\psi}(x)=Sx+\hat{b}$, for all $x\in \ell^2(\mathbb{N})$. Therefore, the corresponding composition operator $C_{\hat{\psi}}$ is bounded on $\mathscr{H}(\ell^2(\mathbb{N}))$.

Note that
$$C_{\hat{\psi}} K_{\hat{\psi}(e_2)}-\left\|C_{\hat{\psi}}\right\|^2K_{e_2}=e^{\left\|\hat{b}\right\|^2+\langle S^*\hat{b},e_2\rangle}K_{S^*Se_2} K_{S^*\hat{b}}-e^{1+\frac{1}{\mu^2}}K_{e_2}=e^{\frac{1}{\mu^2}}K_{e_2}K_{\sqrt{1-\mu^2}e_1}-e^{1+\frac{1}{\mu^2}}K_{e_2},$$ and 
$$C_{\hat{\psi}} K_{\hat{\psi}(e_3)}-\left\|C_{\hat{\psi}}\right\|^2K_{e_3}=e^{\left\|\hat{b}\right\|^2+\langle S^*\hat{b},e_3\rangle}K_{S^*Se_3} K_{S^*\hat{b}}-e^{1+\frac{1}{\mu^2}}K_{e_3}=e^{\frac{1}{\mu^2}}K_{e_3}K_{\sqrt{1-\mu^2}e_1}-e^{1+\frac{1}{\mu^2}}K_{e_3}.$$ Therefore, using Equation \ref{Equtn_inner product}, we conclude that $C^*_{\varphi}$ does not attains norm at $\frac{K_{e_2}+K_{e_3}}{\left\|K_{e_2}+K_{e_3}\right\|}$ even though $\left\|\hat{\psi}(e_2)\right\|=\left\|\hat{\psi}(e_3)\right\|$.

\end{example}

\begin{center}
	\textbf{Acknowledgements}
\end{center}

We are grateful to our PhD supervisors Professor G. Ramesh and Professor D. Venku Naidu, IIT Hyderabad for raising this question. The research of the second named author is supported by the NBHM postdoctoral fellowship, Department of Atomic Energy (DAE), Government of India (File No: 0204/16(21)/2022/R\&D-II/11995).


\begin{thebibliography}{}
	\bibitem{ACOSTA} M. Acosta, R. Aron, D. Garc\'{i}a and M. Maestre, {\em The Bishop-Phelps-Bollob\'{a}s theorem for operators}, J. Funct. Anal. 254 (2008), 2780–2799.
	
	\bibitem{Appel:Bourdon}
	M. J. Appel, P. S. Bourdon\ and\ J. J. Thrall, Norms of composition operators on the Hardy space, Experiment. Math. {\bf 5} (1996), no.~2, 111--117.
		\bibitem{AL} R. Aron and V. Lomonosov, {\em After the Bishop–Phelps theorem}, Acta et Commentationes Universitatis Tartuensis de Mathematica. 18 (2014), 39-49.

	\bibitem{Bonet:Wolf}
		J. Bonet, M. Lindstr\"{o}m\ and\ E. Wolf, Norm-attaining weighted composition operators on weighted Banach spaces of analytic functions, Arch. Math. (Basel) {\bf 99} (2012), no.~6, 537--546.
		\bibitem{CASCALES} B. Cascales, A. Guirao and V. Kadets, {\em A Bishop-Phelps-Bollob\'{a}s type theorem for uniform algebras}, Adv. Math. 240 (2013), 370–382.
	\bibitem{carswell et al.}
	B. J. Carswell, B. D. MacCluer\ and\ A. Schuster, Composition operators on the Fock space, Acta Sci. Math. (Szeged) {\bf 69} (2003), no.~3-4, 871--887.
	\bibitem{CHOI}Y. S. Choi\ and\ S. K. Kim, The Bishop-Phelps-Bollob\'{a}s theorem for operators from $L_1(\mu)$ to Banach spaces with the Radon-Nikod\'{y}m property, J. Funct. Anal. {\bf 261} (2011), no.~6, 1446--1456.
\bibitem{Cowen book}
C. C. Cowen\ and\ B. D. MacCluer, {\it Composition operators on spaces of analytic functions}, Studies in Advanced Mathematics, CRC Press, Boca Raton, FL, 1995. 
\bibitem{Hammond:FLT}
C. Hammond, On the norm of a composition operator with linear fractional symbol, Acta Sci. Math. (Szeged) {\bf 69} (2003), no.~3-4, 813--829. 
\bibitem{Hammond:DS}
C. Hammond, The norm of a composition operator with linear symbol acting on the Dirichlet space, J. Math. Anal. Appl. {\bf 303} (2005), no.~2, 499--508. 
\bibitem{IST}V. Istr\u{a}\c{t}escu, On some hyponormal operators, Pacific J. Math. {\bf 22} (1967), 413--417.
	\bibitem{LEE} Lee, Jun Ik. "On the norm attaining operators." The Korean Journal of Mathematics 20.4 (2012): 485-491..
\bibitem{KOV1}J. Kover, Compact perturbations and norm attaining operators, Quaest. Math. {\bf 28} (2005), no.~4, 401--408.

\bibitem{KOV2}J. Kover, Perturbations by norm attaining operators, Quaest. Math. {\bf 30} (2007), no.~1, 27--33.
	\bibitem{Tle}
	T. Le, Composition operators between Segal-Bargmann spaces, J. Operator Theory {\bf 78} (2017), no.~1, 135--158.
	\bibitem{Martin}
	M. J. Mart\'{\i}n, Norm-attaining composition operators on the Bloch spaces, J. Math. Anal. Appl. {\bf 369} (2010), no.~1, 15--21.
	\bibitem{Motes}
	A. Montes-Rodr\'{\i}guez, The Pick-Schwarz lemma and composition operators on Bloch spaces, Rend. Circ. Mat. Palermo (2) Suppl. {\bf 1998}, no.~56, 167--170. 
	
	\bibitem{Paulsen}
	V. I. Paulsen\ and\ M. Raghupathi, {\it An introduction to the theory of reproducing kernel Hilbert spaces}, Cambridge Studies in Advanced Mathematics, 152, Cambridge University Press, Cambridge, 2016.
	
	\bibitem{Ramesh:Sudip:Venku}
	G. Ramesh, B. Sudip Ranjan, and D. Venku Naidu, Cyclic composition operators on Segal-Bargmann space. {\em Concr. Oper.}, 9(1):127--138, 2022.
	
	\bibitem{Kehe Zhu book}
	K. Zhu, {\it Analysis on Fock spaces}, Graduate Texts in Mathematics, 263, Springer, New York, 2012.
	
	

	
	
	
	
	
	
	
	

	
	
	

	
\end{thebibliography}
\end{document}